\documentclass[lettersize,journal]{IEEEtran}
\usepackage{cite}
\usepackage{amsmath,amssymb,amsfonts}
\usepackage{graphicx}
\usepackage{textcomp}
\usepackage{float}
\usepackage{multicol}
\usepackage{amsmath}
\usepackage{amsthm}
\usepackage{amssymb}
\usepackage{amsfonts}
\usepackage{graphicx}
\usepackage{subfigure}
\usepackage{url}
\usepackage{booktabs}
\usepackage{float}
\usepackage{color}
\usepackage{bm}
\usepackage{algorithm}
\usepackage{algpseudocode}

\newcommand{\col}{\hbox{col}}
\newtheorem{assmp}{\bf Assumption}

\newtheorem{rem}{\bf Remark}

\newtheorem{lem}{\bf Lemma}

\newtheorem{thm}{\bf Theorem}

\newtheorem{defin}{\bf Definition}

\newtheorem{problem}{\bf Problem}

\newcommand{\EQ}{\begin{eqnarray}}
	\newcommand{\EN}{\end{eqnarray}}
\newcommand{\EQQ}{\begin{eqnarray*}}
	\newcommand{\ENN}{\end{eqnarray*}}

\ifCLASSINFOpdf
\else
\fi

\begin{document}
	\title{Data-Driven Cooperative Output Regulation via Distributed Internal Model}

	\author{Liquan~Lin~and~Jie~Huang,~\IEEEmembership{Life Fellow,~IEEE}
		\thanks{The work described in this paper was  supported by a
grant from the Research Grants Council of the Hong Kong Special Administrative
Region, China [Project No.: CUHK 14203924].}
\thanks{The authors are  the Department of Mechanical and Automation Engineering, The Chinese University of Hong Kong, Hong Kong (e-mail: lqlin@mae.cuhk.edu.hk; jhuang@mae.cuhk.edu.hk. Corresponding author: Jie Huang.)}}
	

	\maketitle
	
	\begin{abstract}
The existing result on the cooperative output regulation problem for unknown linear multi-agent systems using a data-driven distributed internal model approach is
limited to the case where each follower  is a single-input and single-output system and the communication network among all agents is an  acyclic static digraph.
In this paper, we further address the same problem for unknown linear multi-agent  systems with multi-input and multi-output followers over a general static and connected digraph. Further we make two main improvements over the existing result. First, we derive a set of much simplified linear systems to be applied by the integral reinforcement learning technique. Thus, the number of the unknown variables governed by a sequence of linear algebraic equations  is much smaller than that of the existing approach. Second, we show that the sequence of linear algebraic equations can be further decoupled to two sequences of linear algebraic equations.
As a result of these two improvements, our approach not only drastically reduces the computational cost, but also significantly weakens the solvability conditions in terms of the the full column rank requirements for these equations.
	
	\end{abstract}
	

	\IEEEpeerreviewmaketitle

    \section{Introduction}
    The past two decades have witnessed flourishing research activities on cooperative control of multi-agent systems (MASs) in various topics such as consensus \cite{jadbabaie2003coordination}\cite{Olfati-Saber-Fax-Murray-2007}\cite{Ren-Beard-2008book}\cite{ma2023}, formation \cite{Olfati-Saber02}\cite{ahn2015}, the cooperative output regulation \cite{su2011}\cite{su2012general} \cite{cai2022cooperative} etc. In particular,
    the cooperative output regulation problem (CORP) has garnered significant attention  due to its immense  applications. The cooperative output regulation problem is a generalization of the output regulation problem for a single-agent system to  a multi-agent system. There are two approaches  for dealing with the output regulation problem (ORP): feedforward  design and  internal model principle \cite{Da}\cite{FW}\cite{Fr}\cite{huang2004}.
   Extending these two approaches to multi-agent systems leads to the so-called distributed-observer approach  \cite{su2011} \cite{cai2022cooperative} and distributed-internal-model approach \cite{su2012general}  \cite{cai2022cooperative}.
   A drawback of the feedforward design and hence the distributed observer approach is that they both make use of the feedforward control, which relies on the solution of the regulator equations, and hence requires the exact system model.

    Recently, there is a surge of interest in studying both ORP for unknown linear systems and CORP for unknown linear multi-agent systems, thus leading to the so-called data-driven approach, which integrates the reinforcement learning (RL) \cite{sutton2018reinforcement}\cite{bertsekas2019reinforcement}, also called adaptive dynamic programming (ADP), technique and the feedforward  design or internal model principle for dealing with ORP, or integrates the reinforcement learning  technique and the
    distributed-observer approach or the distributed-internal-model approach to deal with CORP.

     The reinforcement learning technique has broad applications in many areas. Its application to solving the LQR problem for unknown continuous-time linear systems
     may be traced back to  \cite{vrabie2009adaptive} which studied the LQR problem for linear systems with unknown system matrix by  presenting a policy-iteration (PI) method for solving an algebraic Riccati equation online. Reference \cite{jiang2012computational} further extended the result of \cite{vrabie2009adaptive} to  linear systems with unknown system matrix and input matrix by the PI-based method. However, the PI-based method  requires an initially stabilizing feedback gain to start the iteration process.  Reference \cite{bian2016} broke this barrier by proposing a value-iteration (VI) method to tackle the same  problem as that in \cite{jiang2012computational}, which can start the iteration without an initially stabilizing feedback gain.
     On the basis of the approaches of \cite{jiang2012computational} and \cite{bian2016},   references \cite{gao2016adaptive} and \cite{jiang2022value} studied the optimal output regulation problem for unknown linear systems by combining the feedforward  design with the PI-based method and VI-based method, respectively. Reference \cite{lin2024arxiv}
     further considered the optimal output regulation problem for unknown linear systems by combining the internal model principle with the VI-based method.

     By integrating the technique presented in \cite{gao2016adaptive} and the distributed observer detailed in \cite{su2011}, reference \cite{gao2017cooperative} considered the optimal CORP for linear multi-agent system with unknown state equations. Reference \cite{lin2023}  improved the approach in \cite{gao2017cooperative} by reducing the computational cost and weakening the solvability conditions. Reference \cite{lin2024} further considered the optimal CORP for linear multi-agent system with unknown state equations over jointly connected switching networks. 
     Due to the use of feedforward control, the approaches in \cite{gao2017cooperative}, \cite{lin2023} and \cite{lin2024} require the output equation to be known.
     To deal with the unknown output equation,
     reference \cite{gao2021reinforcement} developed both PI-based and VI-based methods to handle the optimal CORP by combining the distributed-observer  and distributed-internal-model approach.

    However,  the approach in  \cite{gao2021reinforcement} only applies to multi-agent system with  single-input and single-output (SISO) followers over  acyclic communication networks. In this paper, we will further consider the CORP  for multi-agent system with  multi-input and multi-output (MIMO) followers without assuming the network is acyclic. Moreover, all the data-driven approaches to ORP or CORP of linear systems  boils down to solving a sequence of linear algebraic equations governing a set of unknown variables. Thus, the computational cost and the solvability conditions depend on the complexity of these equations. 
    For this reason, we will further make two main improvements over the existing results. First,
     due to the use of the virtual tracking error,  the closed-loop systems of  various followers in \cite{gao2021reinforcement} are coupled. Thus, the learning process in \cite{gao2021reinforcement} is quite complicated, and moreover, it yields a set of high dimensional linear algebraic equations. 
     In contrast, by normalizing the virtual error and making some additional manipulations, we fully decouple the closed-loop systems  of various followers, and thus lead to a set of linear algebraic equations with much lower dimension.
     We further show that the unknown variables of these linear equations can be decomposed into two groups and thus 
     the original derived sequence of linear equations can be reduced to two sets  of lower-dimensional linear algebraic equations. 
     As a result of these two improvements, our approach not only drastically reduces the computational cost, but also significantly weakens the solvability conditions in terms of the  full column rank requirements for these equations.

     The rest of this paper is organized as follows. Section \ref{sec2} provides preliminaries for the CORP based on the distributed internal model method. {\color{black}Section \ref{sec3} presents a novel PI-based approach and a novel VI-based approach for solving the CORP.} Section \ref{sec4} further presents an improved PI-based algorithm and an improved VI-based approach.  Section \ref{sec6} closes the paper with some concluding remarks.

    \indent \textbf{Notation} Throughout this paper, $\mathbb{R}, \mathbb{N} , \mathbb{N}_+$ and $ \mathbb{C}_-$ represent the sets of real numbers, nonnegative integers,  positive integers and the open left-half complex plane, respectively. $\mathcal{P}^n$ is the set of all $n\times n$ real, symmetric and positive semidefinite matrices. $|\cdot|$ represents the Euclidean norm for vectors and the induced norm for matrices. $\otimes$ denotes the Kronecker product.  For $b=[b_1, b_2, \cdots, b_n]^T\in \mathbb{R}^n$, $\text{vecv}(b)=[b_1^2,b_1b_2,\cdots,b_1b_n,b_2^2,b_2b_3,\cdots, b_{n-1}b_n,b_n^2]^T \in \mathbb{R}^{\frac{n(n+1)}{2}}$. For a symmetric matrix $P=[p_{ij}]_{n\times n}\in \mathbb{R}^{n\times n}$, $\text{vecs}(P)=[p_{11},2p_{12},\cdots,2p_{1n},p_{22},2p_{23},\cdots, 2p_{n-1,n}, p_{nn}]^T\in \mathbb{R}^{\frac{n(n+1)}{2}}$. For  $v\in \mathbb{R}^n$, $|v|_P=v^TPv$. For column vectors $a_i, i=1,\cdots,s$,  $\mbox{col} (a_1,\cdots,a_s )= [a_1^T,\cdots,a_s^T  ]^T,$ and, if  $A = (a_1,\cdots,a_s )$, then  vec$(A)=\mbox{col} (a_1,\cdots,a_s )$.
    For $A\in \mathbb{R}^{n\times n}$, $\sigma(A) $ denotes the spectrum of $A$. `blockdiag' denotes the block diagonal matrix operator. $I_n $ denotes the identity matrix of dimension $n$.

    \section{Preliminary} \label{sec2}
    In this section, we review some results on the cooperative output regulation problem  via the internal model approach based on \cite{su2012general} \cite{cai2022cooperative}.

    Consider a class of linear MASs in the following form.
    \begin{align}
    \begin{split}\label{follower}
    	\dot{x}_i&=Ax_i+Bu_i+E_iv\\
    	y_i&=Cx_i
    \end{split}
    \end{align}
    where,  for $i=1,2,\cdots,N$,  $x_i\in \mathbb{R}^n, y_i \in \mathbb{R}^p$ and $u_i\in \mathbb{R}^m$ are the state, measurement output and control input of the $i^{th}$ agent, and $v\in \mathbb{R}^q$ is the state of the exosystem as follows.
    \begin{align}\label{leader}
    	\dot{v}=Sv
    \end{align}
    We define the tracking error of the $i^{th}$ subsystem as follows.
    \begin{align} \label{error}
    	e_i=y_i-y_0=C x_i+Fv
    \end{align}
    where $y_0=-Fv$ is the reference signal.

    Like in \cite{su2012general},   we treat the exosystem \eqref{leader}  and the plant \eqref{follower} as  a leader-follower multi-agent system with $N+1$ agents where \eqref{leader} is the leader and the $N$ subsystems of \eqref{follower} are $N$ followers.  To describe the communication constraints of this leader-follower system, we define a digraph $\mathcal{G}=\{\mathcal{V}, \mathcal{E} \}$ where $\mathcal{V}=\{0,1,\cdots, N \}$ is the node set and $\mathcal{E} \subset \mathcal{V} \times \mathcal{V}$ is the edge set.
    The node $0$ denotes the leader/exosystem \eqref{leader} and the remaining $N$ nodes denote the $N$ followers of (\ref{follower}).
    For $i,j = 0, 1, \cdots,N$, $i \neq j$,  an edge from node $j$ to node $i$ is denoted by $(j, i)$, {\color{black} where the nodes $j$ and $i$ are called the parent node and the child node of each other.}
    {\color{black} A directed spanning tree is a digraph in which every node has exactly one parent
    	except for one node, called the root, which has no parent and from
    	which every other node is reachable. }
    Let $\mathcal{N}_i = \{j | (j,i)\in \mathcal{E}\}$ which is called the neighbor set of node $i$, and let $\mathcal{N}_i^- = \mathcal{N}_i/\{0\}$. Let $\mathcal{A}=[a_{ij}]\in \mathbb{R} ^{(N+1) \times (N+1) }$ be such that $a_{ij}>0$ if $(j,i)\in \mathcal{E} $ and otherwise $a_{ij}=0$, which is called
    the adjacency matrix of $\mathcal{G}$.
    We assume $a_{0j}=0$ for $j=1,2,\cdots,N$  since there are no edges from any follower to the leader.
	$H=[h_{ij}]\in \mathbb{R}^{N\times N}, i=1,2,\cdots,N, j=1,2,\cdots,N$ is defined by $h_{ii}=\sum_{j=0}^{N} a_{ij}$ and $h_{ij}=-a_{ij}$ for all $i\neq j$.
	
	Like in \cite{su2012general}, we describe the cooperative output regulation problem (CORP) as follows.
	\begin{problem}\label{p1}
		Design a distributed controller $u_i$ such that the closed-loop system is exponentially stable with $v$ set to zero and, for $i=1,\cdots,N$, $\lim\limits_{t\to \infty}e_i(t)=0$.
	\end{problem}
	
	Some standard assumptions needed for solving Problem \ref{p1}  are listed as follows.
	
	\begin{assmp}\label{ass2}
		The pair $(A,B)$ is stabilizable.
	\end{assmp}
	
	\begin{assmp}\label{ass3}
		All the eigenvalues of $S$ have nonnegative real parts.
	\end{assmp}
	
	\begin{assmp}\label{ass4}
		For all $\lambda \in \sigma(S)$, $\textup{rank}(\begin{bmatrix}
			A-\lambda I & B\\C & 0
		\end{bmatrix})=n+p$.
	\end{assmp}
	\begin{assmp}\label{ass5}
		The digraph $\mathcal{G}$ contains a directed spanning tree with the node 0 as its root.
	\end{assmp}

	In order to make use of the internal model principle,  we present the concept of minimal $p-copy$ internal model of the  matrix $S$ as follows.
	\begin{defin} \color{black}
		A pair of matrices $(G_1,G_2)$ is said to incorporate the minimum $p-\textup{copy}$ internal model of the matrix $S$ if
		\begin{align}\label{internalmodel}
			G_1=\textup{blockdiag}\underbrace{[\beta,\cdots,\beta]}_{p-tuple}, \; G_2=\textup{blockdiag}\underbrace{[\sigma,\cdots,\sigma]}_{p-tuple}
		\end{align}
		where  $\beta$ is a constant square matrix whose characteristic polynomial equals the minimal polynomial of $S$, and $\sigma$ is a constant column vector such that $(\beta, \sigma)$ is controllable.
	\end{defin}

	To describe our distributed control law, we introduce the virtual tracking error $e_{vi}(t)$ for the $i^{th}$ subsystem as follows  \cite{su2012general}.
\begin{align} \label{eqevi}
		e_{vi}=\sum_{j\in \mathcal{N}_i} a_{ij}(y_i-y_j)
	\end{align}
	
Then,  the distributed dynamic state feedback control law proposed in \cite{su2012general} is as follows.
	\begin{align}
		\begin{split} \label{controller}
			u_i=&K_{x}(\sum_{j\in \mathcal{N}_i^-}a_{ij}(x_i-x_j)+a_{i0}x_i)+K_{z}z_i \\
			\dot{z}_i&=G_1z_i+G_2e_{vi}\\
\end{split}
	\end{align}
	where $K_{x},K_{z}$ are some constant matrices to be designed, and $z_i \in \mathbb{R}^{n_z}$ for some integer $n_z$.

The solvability of Problem \ref{p1} was first studied in \cite{su2012general} and further refined in \cite{cai2022cooperative}. The following result taken from \cite{su2012general} summarizes
the solvability of Problem \ref{p1}.

\begin{lem}
    	 Under Assumptions  \ref{ass2}-\ref{ass5}, there exists a control gain
     $(K_x,K_z)$  such that the following matrix
    	\begin{align}\label{ac}
    		A_c\triangleq \begin{bmatrix}
    			I_N\otimes A+H\otimes (BK_x) & I_N\otimes (BK_z) \\
    			H\otimes (G_2C) & I_N\otimes G_1
    		\end{bmatrix}
    	\end{align}
  is Hurwitz, and moreover, any  control gain $(K_x,K_z)$  that makes  $A_c$ Hurwitz is such that \eqref{controller} solves Problem \ref{p1}.
   \end{lem}

\begin{rem}\label{rem2}
Let $Y\triangleq\begin{bmatrix}
		A&0\\G_2C &G_1
	\end{bmatrix}, J\triangleq\begin{bmatrix}
		B\\0
	\end{bmatrix}$. Under Assumption \ref{ass2}-\ref{ass4}, by Lemma 1.26 in \cite{huang2004},  $(Y,J)$ is stabilizable. Thus, the Riccati Equation
	\begin{align}\label{Ricca}
		Y^TP+PY-PJJ^TP+I_{n+n_z}=0
	\end{align}
	admits a unique positive definite solution $P^*$.
	 By  Lemma 9.3 in \cite{cai2022cooperative}, any control gain	 \begin{align}\label{gain}
	 	[K_x,K_z]\triangleq K=-\nu^{-1} J^TP^*
	 \end{align}
	 with $0<\nu\leq  2\text{Re}(\lambda_i), \forall \lambda_i\in \sigma(H), i=1,\cdots,N$, makes  $A_c$ Hurwitz.
	 \end{rem}

	We now summarize the main result of  \cite{su2012general} as follows.
	 \begin{thm}
	 		 Under Assumptions \ref{ass2}-\ref{ass5}, let  $(G_1,G_2)$ incorporate a minimum $p-\textup{copy}$ internal model of the matrix $S$.
Then Problem \ref{p1} is solved  by the distributed dynamic feedback control law \eqref{controller} with the feedback control gain \eqref{gain}.
	 \end{thm}

\begin{rem}
		Two significant differences between the method of \cite{gao2021reinforcement} and the method here are worth mentioning.
First, the internal model used in \cite{gao2021reinforcement} is 1-copy which is a special case of that used here with $G_2=[0,0,\cdots,0,1]^T$ be a column vector. Thus,  every subsystem of the follower system
in  \cite{gao2021reinforcement} must be SISO. In contrast, each subsystem of the follower system
here can be MIMO. Second, the approach of \cite{gao2021reinforcement} only applies to the case where the matrix  $A_c$ in  equation (16) of \cite{gao2021reinforcement} is lower triangular, which is possible only if the communication graph is  acyclic.
In contrast, Assumption \ref{ass5} here accommodates  graphs with cycles.
\end{rem}

  \section{Adaptive Cooperative Output Regulation with Normalized Virtual Tracking Error} \label{sec3}

     In this section, we further consider solving Problem \ref{p1} without knowing
      the matrices $A,B,C,E_i$ and $F$ by data-driven approaches.  The key for solving this problem is to obtain the solution of the Riccati equation \eqref{Ricca} without knowing
      the matrices $A,B,C,E_i$ and $F$.
    For this purpose, we need to relate the system's dynamics to the Riccati equation to be solved.
     Then, with the virtual error $e_{vi}$ given by \eqref{eqevi}, we can obtain the following so-called augmented system:

	\begin{align} \label{sysdy}
		\begin{split}
			\dot{x}_i&=Ax_i+Bu_i+E_iv\\
			\dot{z}_i&=G_1z_i+G_2 \sum_{j\in \mathcal{N}_i}a_{ij}(y_i-y_j)
			\end{split}
\end{align}
Let $\xi_i\triangleq \col (x_i,z_i)$.  Then
     $\xi_i$ is governed by the following differential equation
     \begin{align}\label{dyxi}
     	\dot{\xi}_i= & \hat{Y} \xi_i+Ju_i+\begin{bmatrix}
     		E_i\\0
     	\end{bmatrix}v+\gamma_i
     \end{align}
     where $ \hat{Y}  = \begin{bmatrix}
     		A&0\\ (\sum_{j\in \mathcal{N}_i} a_{ij}) G_2C & G_1
     	\end{bmatrix}$ and   $\gamma_i=-\begin{bmatrix}
     	0\\ \sum_{j\in \mathcal{N}_i}( a_{ij} G_2y_j)
     \end{bmatrix}$. Since $\hat{Y}  \neq Y$,  we cannot relate \eqref{dyxi} to the Riccati equation \eqref{Ricca}.
     To fix this problem, let
      \begin{align}\label{eqhevi}
      \hat{e}_{vi}=\frac{e_{vi}}{\sum_{j\in \mathcal{N}_i} a_{ij}}
    \end{align}
  which is    the normalized version of  $e_{vi}$. Then we obtain the following modified augmented system
    \begin{align} \label{sysdyhat}
     	\begin{split}
     		\dot{x}_i&=Ax_i+B{u}_i+E_iv\\
     		\dot{\hat{z}}_i&=G_1\hat{z}_i+G_2\hat{e}_{vi}\\
     		\hat{e}_{vi}&=\frac{\sum_{j\in \mathcal{N}_i}a_{ij}(y_i-y_j)}{\sum_{j\in \mathcal{N}_i} a_{ij}}
     	\end{split}
     \end{align}
     and the modified distributed control law:
     \begin{align}
     	\begin{split} \label{modicontroller}
     		{u}_i=&\frac{K_{x}}{\sum_{j\in \mathcal{N}_i} a_{ij}}(\sum_{j\in \mathcal{N}_i^-}a_{ij}(x_i-x_j)+a_{i0}x_i)+K_{z}\hat{z}_i
     	\end{split}
     \end{align}

     Let $x=\col (x_1,\cdots,x_N), u=\col (u_1,\cdots,u_N),  \hat{e}_v=\col (\hat{e}_{v1} ,\cdots,\hat{e}_{vN} ) , \hat{z}=\col (\hat{z}_1 ,\cdots,\hat{z}_N ), \tilde{A}=I_N\otimes A, \tilde{B}=I_N\otimes B, \tilde{E}=\col (E_1 ,\cdots, E_N), \tilde{G}_1=I_N\otimes G_1, \tilde{G}_2=I_N\otimes G_2, \hat{C}=(DH)\otimes C, \hat{F}=(DH\bold{1}_N)\otimes F$, and $D=\text{blockdiag}(\frac{1}{\sum_{j\in \mathcal{N}_1} a_{1j}}, \frac{1}{\sum_{j\in \mathcal{N}_2} a_{2j}}, \cdots, \frac{1}{\sum_{j\in \mathcal{N}_N} a_{Nj}})$.
	Then, the compact form of \eqref{sysdyhat} is given by
     \begin{align}\label{sysdycomhat}
     	\begin{split}
     		\dot{x}&=\tilde{A}x+\tilde{B}{u}+\tilde{E}v\\
     	\dot{\hat{z}}&=\tilde{G}_1\hat{z}+\tilde{G}_2\hat{e}_v\\
     	\hat{e}_v&=\hat{C}x+\hat{F}v
     	\end{split}
     \end{align}

     We have the following lemma regarding the solvability of the cooperative output regulation problem of \eqref{sysdycomhat} by control law \eqref{modicontroller}.

     \begin{lem} \label{lem2}
     	Under Assumptions \ref{ass2}-\ref{ass5},
      the cooperative output regulation problem of \eqref{sysdycomhat} is solved by the distributed controller \eqref{modicontroller} if $(K_x,K_z)$ is such that the following matrix $\hat{A}_c$ is Hurwitz.
     	\begin{align}\label{hatac}
     		\hat{A}_c\triangleq \begin{bmatrix}
     			I_N\otimes A+(DH)\otimes (BK_x) & I_N\otimes (BK_z) \\
     			(DH)\otimes (G_2C) & I_N\otimes G_1
     		\end{bmatrix}
     	\end{align}
     \end{lem}

     \begin{proof}
     	The compact form of  \eqref{modicontroller} is as follows:
     \begin{align}\label{hatu}
     	{u}=((DH)\otimes K_x)x+(I_N\otimes K_z)\hat{z}
     \end{align}
Substituting     \eqref{hatu}   to \eqref{sysdycomhat} gives the following closed-loop system:

 \begin{align}\label{closedcomhat}
     	\begin{split}
     			\dot{x}&=(\tilde{A}+\tilde{B}((DH)\otimes K_x) )x+\tilde{B}(I_N\otimes K_z)\hat{z }+\tilde{E}v\\
     	\dot{\hat{z}}&=\tilde{G}_2\hat{C}  x + \tilde{G}_1 \hat{z} + \tilde{G}_2\hat{F}v\\
     	\hat{e}_v&=\hat{C}x+\hat{F}v
     	\end{split}
     \end{align}

 Let \begin{align} \label{eqhatac}
     	\hat{A}_c= \begin{bmatrix}
     		\tilde{A}+\tilde{B}((DH)\otimes K_x) & \tilde{B}(I_N\otimes K_z) \\
     		\tilde{G}_2\hat{C} & \tilde{G}_1
     	\end{bmatrix}
     \end{align}
which is the system matrix of the closed-loop system \eqref{closedcomhat} and is the same as the matrix on the right hand side of \eqref{hatac}.
Suppose $(K_x,K_z)$ is such that $\hat{A}_c$ is Hurwitz. Since $(\tilde{G}_1,\tilde{G}_2)$ incorporates a minimum $Np-\textup{copy}$ internal model of the matrix $S$, by Lemma 1.27 of \cite{huang2004}, for any $\tilde{E}$ and $\hat{F}$, the following  matrix equations have a unique solution $(X,Z)$:
     \begin{align}
     	\begin{split}
     		XS&=\tilde{A}X+\tilde{B}[((DH)\otimes K_x)X+(I_N\otimes K_z)Z]+\tilde{E}\\
     		ZS&=\tilde{G}_1Z+\tilde{G}_2(\hat{C}X+\hat{F})\\
     		0&=\hat{C}X+\hat{F}
     	\end{split}
     \end{align}

      Let $\beta_x=x-Xv,\beta_z=\hat{z}-Zv$. Then it can be verified that
     \begin{align*}
     	\begin{bmatrix}
     		\dot{\beta}_x \\ \dot{\beta}_z
     	\end{bmatrix}&=\hat{A}_c\begin{bmatrix}
     	\beta_x \\ \beta_z
     	\end{bmatrix}\\
     	\hat{e}_v&=\hat{C}\beta_x
     \end{align*}
  Since $\hat{A}_c$ is Hurwitz,     $\lim\limits_{t\to \infty}\hat{e}_v=0$.  Note that $\hat{e}_v=(D H\otimes I_p)e$ with $e= \col (e_1,\cdots,e_N)$.
  It can be seen that, under Assumption \ref{ass5},  $D$ is nonsingular. Also, by Lemma 4 of \cite{hu2007},  under Assumption \ref{ass5},  all the eigenvalues of the matrix $H$ have positive real parts.    
  Thus,  we have $\lim_{t\to \infty}e=0$ if  $\lim_{t\to \infty}\hat{e}_v=0$.
%
 \end{proof}

       We are now ready to present the following result.

       \begin{thm}\label{thm2}
       	 Under Assumptions \ref{ass2}-\ref{ass5}, Problem \ref{p1} is solved  by the distributed dynamic feedback control law \eqref{modicontroller} with the following feedback control gain:
       	\begin{align}\label{gainrevi}
       			[K_x,K_z]\triangleq K=-\omega^{-1} J^TP^*
       	\end{align}
       	where $P^*$ is the solution to the Riccati equation \eqref{Ricca} and  $0<\omega\leq 2\text{Re}(\lambda_i), \forall \lambda_i\in \sigma(DH), i=1,\cdots,N$.
       \end{thm}

       \begin{proof}
       	The Theorem has been established for the special case where $D$ is an identity matrix in Lemma 9.3 of \cite{cai2022cooperative}.
       Since the only property of the matrix $H$ used in the proof of Lemma 9.3 of \cite{cai2022cooperative} is that all the eigenvalues of the matrix $H$ have positive real parts which,  under Assumption \ref{ass5},  is guaranteed by Lemma 4 of \cite{hu2007}. Since, under Assumption \ref{ass5}, $D$ is a diagonal matrix with all diagonal elements positive,
        all the eigenvalues of the matrix $DH$ have positive real parts. Thus, the proof of Theorem \ref{thm2} is the same as that of Lemma 9.3 of \cite{cai2022cooperative} upon replacing $H$ by $DH$ everywhere.

      \end{proof}

     Let $\hat{\xi}_i\triangleq \col (x_i,\hat{z}_i)$. Then, under the control law \eqref{modicontroller},
     $\hat{\xi}_i$ is governed by the following differential equation
     \begin{align}\label{dyhatxi}
       	\dot{\hat{\xi}}_i= &Y\hat{\xi}_i+Ju_i+\begin{bmatrix}
       		E_i\\0
       	\end{bmatrix}v+\hat{\gamma}_i
       \end{align}
     where $\hat{\gamma}_i=-\begin{bmatrix}
     	0\\ \frac{\sum_{j\in \mathcal{N}_i} (a_{ij} G_2y_j)}{\sum_{j\in \mathcal{N}_i} a_{ij}}
     \end{bmatrix}$ .

       Based on \eqref{dyhatxi}, we can apply the off-policy RL/ADP methods to approximately solve the Riccati equation \eqref{Ricca} and obtain the controller without the information of the system matrices provided that the data of $\hat{\xi}_i,u_i,v$ and $\hat{\gamma}_i$ are available. However, due to the communication constraints, some followers cannot directly obtain the information of the exosystem's state $v$, i.e., agent $i$ can access the signal $v$ if and only if $0\in \mathcal{N}_i$. To circumvent this difficulty,  we will introduce the following distributed observer proposed in \cite{su2011} to estimate $v$:
     \begin{align}\label{observer}
     	\dot{\eta}_i=S\eta_i+\mu\sum_{j\in \mathcal{N}_i}a_{ij}(\eta_j-\eta_i),~~i = 1, \cdots, N
     \end{align}
     where $\eta_i \in \mathbb{R}^q$ for $i = 0, 1, \cdots, N$ with $\eta_0 \triangleq v$, and $\mu>\mu_0>0 $ with $\mu_0$ determined by $S$ and digraph $\mathcal{G}$. By Theorem 1 of \cite{su2011}, $(\eta_i-v)$ tends to 0 exponentially.

\begin{rem}
	{\color{black}
		Instead of \eqref{observer}, reference \cite{gao2021reinforcement} made use of the following adaptive distributed observer proposed in \cite{cai2017} to estimate the state $v$ of the exosystem:
	\begin{align}\label{adaobserver}
		\begin{split}
		\dot{S}_i=&\mu_1\sum_{j\in \mathcal{N}_i}a_{ij}(S_j-S_i)\\
		\dot{\eta}_i=&S_i\eta_i+\mu_2\sum_{j\in \mathcal{N}_i}a_{ij}(\eta_j-\eta_i),~~i = 1, \cdots, N
	\end{split}
	\end{align}
	where $S_0 := S$ and $\mu_1,\mu_2>0$ are some positive constants. It was shown in \cite{cai2017} that,  there exist two real positive numbers $\mu_{10}$ and $\mu_{20}$ determined by the eigenvalues of $S$ and $H$  such that, for all $\mu_{1} > \mu_{10}$ and $\mu_{2} > \mu_{20}$, the solutions  of \eqref{adaobserver} are such that, for $i = 1, \cdots, N$, $S_i$ tend to $S$ and $\eta_i$ tend to $\eta$, exponentially. Thus, an advantage of \eqref{adaobserver} over \eqref{observer} is that it only requires the children of the leader be able to access the matrix $S$. Nevertheless, this advantage is paid by more computational cost of \eqref{adaobserver} over \eqref{observer}.
In this paper, for better elucidating the advantage of our approach, we choose to use  \eqref{observer} instead of \eqref{adaobserver}.
	}
\end{rem}

Now rewrite \eqref{dyhatxi} to the following form:
     \begin{align}\label{dyhatxirhox}
     	\dot{\hat{\xi}}_i= &Y\hat{\xi}_i+Ju_i+\begin{bmatrix}
     		E_i\\0
     	\end{bmatrix}\eta_i+\rho_i+\hat{\gamma}_i
     \end{align}
     where $\rho_i=\begin{bmatrix}
     	E_i\\0
     \end{bmatrix}(v-\eta_i)$ tends to 0 exponentially.

  \begin{rem} \label{rem5}
Two points are worth mentioning. First, $\hat{\gamma}_i $ satisfies the communication constraints and is thus accessible by agent $i$ through the communication network.
Second,  since $\rho_i$ tends to 0 exponentially, we can neglect the term  $\rho_i$ in \eqref{dyhatxirhox} by starting the data collecting process after some finite time $t_0>0$. 
     \end{rem}
     
     For $t\geq t_0$, we can simplify \eqref{dyhatxirhox} to the following form:
     \begin{align}\label{dyhatxirho}
     	\dot{\hat{\xi}}_i= &Y\hat{\xi}_i+Ju_i+\begin{bmatrix}
     		E_i\\0
     	\end{bmatrix}\eta_i+\hat{\gamma}_i,~~ t \geq t_0
     \end{align}

   In what follows, we will  develop both PI-based and VI-based approaches  to solve Problem \ref{p1} based on \eqref{dyhatxirho}.

      \subsection{PI-based Approach}\label{sec3-1}

     Motivated by \cite{jiang2012computational}, we apply the following iteration method \cite{Kleinman}
     to solve the Riccati equation \eqref{Ricca}
     \begin{align}\label{iteraeq1}
     	0&=(Y+JK_k)^TP_k+P_k(Y+JK_k)+I_{n+n_z}+K_k^TK_k\\ \label{iteraeq2}
     	K_{k+1}&=-J^TP_k
     \end{align}
     where $k=0,1,\cdots$, and $K_0$ is such that $Y+JK_0$ is Hurwitz.
     \eqref{iteraeq1} and \eqref{iteraeq2} guarantees the following property for $k\in \mathbb{N}$:
     \begin{enumerate}
     	\item $\sigma(Y+JK_k)\in \mathbb{C}_-$;
     	\item $P^*\leq P_{k+1}\leq P_k$;
     	\item $\lim_{k\to \infty }K_k=-J^TP^*, \lim_{k\to \infty }P_k=P^*$.
     \end{enumerate}

     Let $Y_k=Y+JK_k$. From \eqref{dyhatxirho}, we have
     \begin{align}\label{dyhatxirevi}
     	\dot{\hat{\xi}}_i= &Y\hat{\xi}_i+Ju_i+\begin{bmatrix}
     		E_i\\0
     	\end{bmatrix}\eta_i+\hat{\gamma}_i \notag \\
     	=& Y_k\hat{\xi}_i+J(-K_k\hat{\xi}_i+u_i)+\begin{bmatrix}
     		E_i\\0
     	\end{bmatrix}\eta_i+\hat{\gamma}_i
     \end{align}
     Based on \eqref{iteraeq1}-\eqref{dyhatxirevi}, we obtain
     \begin{align} \label{intequahat}
     	|\hat{\xi}_i&(t+\delta t)|_{P_k}-|\hat{\xi}_i(t)|_{P_k}\notag \\=&\int_{t}^{t+\delta t}\lbrack |\hat{\xi}_i|_{Y_k^TP_k+P_kY_k}+ 2(-K_k\hat{\xi}_i+u_i)^TJ^TP_k\hat{\xi}_i \notag \\&+2\eta_i^T\begin{bmatrix}
     		E_i\\0
     	\end{bmatrix}^TP_k\hat{\xi}_i  +2\hat{\gamma}_i^TP_k\hat{\xi}_i  \rbrack d\tau\notag \\
     	=& \int_{t}^{t+\delta t}[-|\hat{\xi}_i|_{I_{n+n_z}+K_k^TK_k}+2(K_k\hat{\xi}_i-u_i)^TK_{k+1}\hat{\xi}_i\notag \\& +2\eta_i^T\begin{bmatrix}
     		E_i\\0
     	\end{bmatrix}^TP_k\hat{\xi}_i  +2\hat{\gamma}_i^TP_k\hat{\xi}_i]d\tau
     \end{align}

     For convenience, define a constant matrix $M\in \mathbb{R}^{(n+n_z)^2\times \frac{(n+n_z)(n+n_z+1)}{2}}$ such that $M\text{vecs}(Q)=\text{vec}(Q)$ for any symmetric matrix $Q\in \mathbb{R}^{(n+n_z)\times (n+n_z)}$.

     Furthermore,  for any vectors $a\in \mathbb{R}^n,b \in \mathbb{R}^m$ and any integer $s\in \mathbb{N}_+$, define
     \begin{align}
     	\begin{split}\label{newdefi}
     		\delta _a=&[\text{vecv}(a(t_1))-\text{vecv}(a(t_0)), \cdots ,\\ &\text{vecv}(a(t_s))-\text{vecv}(a(t_{s-1}))]^T\\
     		\Gamma_{ab}=&[\int_{t_0}^{t_1}a\otimes b d\tau , \int_{t_1}^{t_2}a\otimes b d\tau, \cdots , \int_{t_{s-1}}^{t_s}a\otimes b d\tau]^T
     	\end{split}
     \end{align}

     Then, from \eqref{intequahat} and \eqref{newdefi}, we obtain
     \begin{align}\label{pilinearhat}
     	\hat{\Psi}_{PI,ik}\begin{bmatrix}
     		\text{vecs}(P_{k})\\
     		\text{vec}(K_{k+1})\\
     		\text{vec}(\begin{bmatrix}
     			E_i\\0
     		\end{bmatrix}^TP_k)
     	\end{bmatrix}=\hat{\Phi}_{PI,ik}
     \end{align}
     where
     \begin{align*}
     	\hat{\Psi}_{PI,ik}=&[\hat{\Psi}_{PI,ik}^1, \hat{\Psi}_{PI,ik}^2, \hat{\Psi}_{PI,ik}^3]\\
     	\hat{\Psi}_{PI,ik}^1=&\delta_{\hat{\xi}_i}-2\Gamma_{\hat{\xi}_{i}\hat{\gamma}_i}M\\
     	\hat{\Psi}_{PI,ik}^2=&-2\Gamma_{\hat{\xi}_i\hat{\xi}_i}(I_{n+n_z}\otimes K_k^T)+2\Gamma_{\hat{\xi}_iu_i}\\
     	\hat{\Psi}_{PI,ik}^3=&-2\Gamma_{\hat{\xi}_i\eta_i}\\
     	\hat{\Phi}_{PI,ik}=&-\Gamma_{\hat{\xi}_i\hat{\xi}_i}\text{vec}(I_{n+n_z}+K_k^TK_k)
     \end{align*}

     Starting with an initially stabilizing gain $K_0$ such that $Y+JK_0$ is Hurwitz, one can iteratively solve \eqref{pilinearhat} to obtain the approximate solution to \eqref{Ricca} and hence the desired feedback control gain from \eqref{gainrevi}.

     We further establish the solvability of  \eqref{pilinearhat} by the following lemma.
     \begin{lem} \label{lem3-1}
     	The matrix $\hat{\Psi}_{PI,ik}$ in \eqref{pilinearhat} has full column rank for $k\in \mathbb{N}$ if
     	\begin{align}\label{rankconpihat}
     		\textup{rank}([\Gamma_{\hat{\xi}_i\hat{\xi}_i}, \Gamma_{\hat{\xi}_iu_i}, \Gamma_{\hat{\xi}_i\eta_i}])=&\frac{(n+n_z)(n+n_z+1)}{2}\notag \\
     		&+(n+n_z)(m+q)
     	\end{align}
     \end{lem}

     \begin{proof}
     	Motivated by \cite{jiang2012computational}, we prove this lemma by contradiction. Assume $\Xi_v=[W_v^T,U_v^T,Z_v^T]^T$ is a nonzero solution to the following equation:
     	\begin{align}\label{l11}
     		\hat{\Psi}_{PI,ik} \Xi_v=0
     	\end{align}
     	
     	There exist $W_m=W_m^T\in \mathbb{R}^{(n+n_z)\times (n+n_z)}, U_m\in \mathbb{R}^{m\times (n+n_z)}$ and $ Z_m\in \mathbb{R}^{q \times (n+n_z)}$, such that $W_v=\text{vecs}(W_m), U_v=\text{vec}(U_m),Z_v=\text{vec}(Z_m)$. Based on (\ref{dyhatxirevi}), (\ref{intequahat}) and (\ref{pilinearhat}), we know
     	\begin{align}
     		\begin{split}\label{l12}
     			\hat{\Psi}_{PI,ik} \Xi_v=& [\hat{\Psi}_{PI,ik}^1, \hat{\Psi}_{PI,ik}^2, \hat{\Psi}_{PI,ik}^3]\begin{bmatrix}
     				\text{vecs}(W_m)\\\text{vec}(U_m)\\\text{vec}(Z_m)
     			\end{bmatrix}\\
     			=&(\delta_{\hat{\xi}_i}-2\Gamma_{\hat{\xi}_{i}\hat{\gamma}_i}M)\text{vecs}(W_m)\\ &+(-2\Gamma_{\hat{\xi}_i\hat{\xi}_i}(I_{n+n_z}\otimes K_k^T)+2\Gamma_{\hat{\xi}_iu_i})\text{vec}(U_m)
     			\\&-2\Gamma_{\hat{\xi}_i\eta_i}\text{vec}(Z_m)\\
     			=& \Gamma_{\hat{\xi}_i\hat{\xi}_i}\text{vec}(Y_k^TW_m+W_mY_k) \\ &-\Gamma_{\hat{\xi}_i\hat{\xi}_i}\text{vec}(K_{k}^TJ^TW_m+W_mJK_{k})\\ &+2\Gamma_{\hat{\xi}_iu_i}\text{vec}(J^TW_m)\\& +2\Gamma_{\hat{\xi}_i\eta_i}\text{vec}(\begin{bmatrix}
     				E_i\\0
     			\end{bmatrix}^TW_m)\\&-\Gamma_{\hat{\xi}_i\hat{\xi}_i}\text{vec}(K_{k}^TU_m+U_m^TK_{k})\\&+2\Gamma_{\hat{\xi}_iu_i}\text{vec}(U_m)-2\Gamma_{\hat{\xi}_i\eta_i}\text{vec}(Z_m)\\
     			=&\Gamma_{\hat{\xi}_i\hat{\xi}_i}\text{vec}(\Omega_1)+2\Gamma_{\hat{\xi}_iu_i}\text{vec}(\Omega_2)+2\Gamma_{\hat{\xi}_i\eta_i}\text{vec}(\Omega_3)\\
     			=&0
     		\end{split}	
     	\end{align}
     	where
     	\begin{align}
     		\begin{split}\label{l13}
     			\Omega_1=&Y_k^TW_m+W_mY_k-K_{k}^T(J^TW_m+U_m)\\&-(W_mJ+U_m^T)K_{k}\\
     			\Omega_2=&J^TW_m+U_m\\
     			\Omega_3=&\begin{bmatrix}
     				E_i\\0
     			\end{bmatrix}^TW_m-Z_m
     		\end{split}
     	\end{align}

     	Noting that $\Omega_1\in \mathbb{R}^{(n+n_z)\times (n+n_z)}$ is symmetric, we have
     	\begin{align}\label{l14}
     		\Gamma_{\hat{\xi}_i\hat{\xi}_i}\text{vec}(\Omega_1)
     		=&\hat{\Gamma}_{\hat{\xi}_{i}}\text{vecs}(\Omega_1)
     	\end{align}
     	where 
     $$\hat{\Gamma}_{\hat{\xi}_i}=[\int_{t_0}^{t_1}\text{vecv}(\hat{\xi}_i) d\tau ,  \cdots , \int_{t_{s-1}}^{t_s}\text{vecv}(\hat{\xi}_i) d\tau]^T$$.
     	
     	Then, (\ref{l12}) implies the following:
     	\begin{align*}
     		[\hat{\Gamma}_{\hat{\xi}_{i}}, 2\Gamma_{\hat{\xi}_iu_i}, 2\Gamma_{\hat{\xi}_i\eta_i}]\begin{bmatrix}
     			\text{vecs}(\Omega_1)\\ \text{vec}(\Omega_2)\\ \text{vec}(\Omega_3)
     		\end{bmatrix}=0
     	\end{align*}
     	
     	Under the full rank condition of (\ref{rankconpihat}), we get
     	\begin{align*}
     		\begin{bmatrix}
     			\text{vecs}(\Omega_1)\\ \text{vec}(\Omega_2)\\ \text{vec}(\Omega_3)
     		\end{bmatrix}=0
     	\end{align*}
     	which implies $U_m=-J^TW_m, Z_m=\begin{bmatrix}
     		E_i\\0
     	\end{bmatrix}^TW_m$ and $Y_k^TW_m+W_mY_k=0$. It is known that $\sigma(Y_k)\subset \mathbb{C}^{-}$ from Property 1 of Kleinman's algorithm \eqref{iteraeq1}-\eqref{iteraeq2}, which means $W_m=0$ is the only solution to the equation $Y_k^TW_m+W_mY_k=0$. Thus, we have $\Xi_v=0$, which contradicts the assumption that $\Xi_v \neq 0$. The proof is completed.
     \end{proof}

\begin{rem}\label{rem3v}
     It is interesting to compare our approach with that of	\cite{gao2021reinforcement}, which studied the cooperative output regulation problem for SISO linear MASs over acyclic graph via a distributed internal model approach. The augmented system of \cite{gao2021reinforcement} takes the following form
   \begin{align}\label{gaoaugmented}
     	\dot{\xi}_i = Y{\xi}_i+Ju_i+\begin{bmatrix}
     		E_i\\\alpha_{i0}G_2F
     	\end{bmatrix}v-\sum_{j\in \mathcal{N}_i^-} \Theta_{ij}\xi_j
     	\end{align}
     	where $\xi_i$ represents the augmented state of agent $i$ in \cite{gao2021reinforcement}, $\Theta_{ij}=\begin{bmatrix}
     		0&0\\\alpha_{ij}G_2C&0
     	\end{bmatrix}$, $\alpha_{ij}=\frac{a_{ij}}{\sum_{j\in \mathcal{N}_i} a_{ij}}$. {\color{black}Here, for fair comparison, we have ignored an error term in \eqref{gaoaugmented} caused by using the adaptive distributed observer \eqref{adaobserver} instead of the distributed observer \eqref{observer}, and set the weighting matrices $Q$ and $R$ in \cite{gao2021reinforcement} to be identity matrices.}

Based on  \eqref{gaoaugmented}, replace $v$ with $\eta_i$ and ignore the error caused by this replacement, the following equations are derived:
     	\begin{align}\label{gaolinear}
     	{\Psi}_{PI,ik}\begin{bmatrix}
     			\text{vecs}(P_{k})\\
     			\text{vec}(K_{k+1})\\
     			\text{vec}(\begin{bmatrix}
     				E_i\\\alpha_{i0}G_2F
     			\end{bmatrix}^TP_k)\\
     				\text{vec}(\Xi_{ik})
     		\end{bmatrix}={\Phi}_{PI,ik}
     	\end{align}
     	where
     	\begin{align*}
     		{\Psi}_{PI,ik}=&[\delta_{\xi_{i}}, -2\Gamma_{{\xi}_i{\xi}_i}(I_{n+n_z}\otimes K_k^T)+2\Gamma_{{\xi}_iu_i}, \\&-2\Gamma_{{\xi}_i\eta_i}, 2\Gamma_{{\xi}_i\psi_i}]\\
     		{\Phi}_{PI,ik}=&-\Gamma_{{\xi}_i{\xi}_i}\text{vec}(I_{n+n_z}+K_k^TK_k)\\
     		\psi_i=&[\xi_{j_1}^T,\xi_{j_2}^T,\cdots,\xi_{j_{|\mathcal{N}_i^-|}}^T]^T\in \mathbb{R}^{h_i}\\
     		\Xi_{ik}=&[P_k\Theta_{ij_1},P_k\Theta_{ij_2},\cdots, P_k\Theta_{ij_{|\mathcal{N}_i^-|}}]^T
     	\end{align*}
  Since the equations in \eqref{gaoaugmented} are coupled, which makes \eqref{gaolinear} much more complicated.
  In particular,  the unknown matrix $\Xi_{ik}$  not only depends on ${\xi}_i$ but also ${\xi}_j$ with $j\in \mathcal{N}_i^-$.  Let  $h_i=|\mathcal{N}_i^-|(n+n_z)$ denotes the dimension of $\psi_i$. Then,
  the number of the unknown variables in \eqref{gaolinear} is $\frac{(n+n_z)(n+n_z+1)}{2} + (n+n_z)(m+q+h_i)$. In contrast, our equation  \eqref{pilinearhat} only contains  $\frac{(n+n_z)(n+n_z+1)}{2} + (n+n_z)(m+q)$ unknown variables, thus reducing the number of the unknown variables by $(n+n_z)h_i$.
  Moreover,  the solvability condition for \eqref{gaolinear}  is as follows
     	\begin{align*}
     		&\textup{rank}([\Gamma_{{\xi}_i{\xi}_i}, \Gamma_{{\xi}_iu_i},  \Gamma_{{\xi}_i\eta_i}, \Gamma_{{\xi}_i\psi_i}])\\=&\frac{(n+n_z)(n+n_z+1)}{2}
     		+(n+n_z)(m+q+h_i)
     	\end{align*}
     	Again, checking these conditions is much more tedious than checking the rank condition \eqref{rankconpihat}.
     \end{rem}

     To get some idea on the degree of the improvement of our PI approach over the  PI algorithm in \cite{gao2021reinforcement},  consider a moderate example with $n=10,m=8,p=4,q=20,n_z=40, N = 5$. Let us assume each follower has  $2$ neighbors, i.e. $|\mathcal{N}_i^-|=2$. Then, $h_i=2(n+n_z)=100$. TABLE \ref{tab11} and TABLE \ref{tab22} show, for each follower,
     the numbers of unknown variables and rank conditions of the two algorithms, respectively.  The contrast is stark.

     \begin{table}[!ht]
     	\caption{Comparison of Computational Complexity}\label{tab11}
     	\center
     	\begin{tabular}{|c|c|c|}\hline

     		Algorithm & Equation Number &Unknown variables \\ \hline
     		\cite{gao2021reinforcement}'s PI Algorithm & \eqref{gaolinear}	& 7675 \\  \hline
     		
     		Our PI Algorithm &\eqref{pilinearhat} & 2675 \\ \hline
     		
     	\end{tabular}
     	
     \end{table}

     \begin{table}[!ht]
     	\caption{Comparison of rank condition}\label{tab22}
     	\center
     	\begin{tabular}{|c|c|}\hline

     		Algorithm & Rank Condition for Data Collecting \\ \hline
     		\cite{gao2021reinforcement}'s PI Algorithm & 	$\textup{rank}([\Gamma_{{\xi}_i{\xi}_i}, \Gamma_{{\xi}_iu_i},  \Gamma_{{\xi}_i\eta_i}, \Gamma_{{\xi}_i\psi_i}])=7675$
     		 \\ \hline
     		
     		Our PI Algorithm& $\textup{rank}([\Gamma_{\hat{\xi}_i\hat{\xi}_i}, \Gamma_{\hat{\xi}_iu_i}, \Gamma_{\hat{\xi}_i\eta_i}])= 2675$   \\  \hline
     		
     	\end{tabular}
     	
     \end{table}

     \subsection{VI-based Approach }\label{sec3-2}

     The PI-based method requires an initially stabilizing feedback gain $K_0$ to start the iteration. Obtaining such a $K_0$ usually requires some prior knowledge of the unknown matrices $Y$ and $J$. For the case where  an initially stabilizing feedback gain $K_0$ is not available, one can resort to the VI-based approach.   Reference  \cite{gao2021reinforcement} also developed a VI-based algorithm for data-driven CORP for MASs with SISO followers over acyclic graphs based on  the algorithm of \cite{bian2016}. In this subsection, we further deal with the more general case with  a simpler and more efficient VI-based algorithm.

    Like \cite{bian2016},   let $\epsilon_k$ be a series of time steps satisfying
     \begin{align}
     	\epsilon_k>0, \; \sum_{k=0}^{\infty}\epsilon_k=\infty, \; \sum_{k=0}^{\infty}\epsilon_k^2<\infty
     \end{align}
     $\{B_q\}_{q=0}^{\infty}$ is a collection of bounded sets satisfying
     \begin{align}
     	B_q\subset B_{q+1}, \; q\in \mathbb{N}, \; \lim\limits_{q\to \infty}B_q=\mathcal{P}^{n+n_z}
     \end{align}
     and $\varepsilon>0$ is the convergence criterion. Then,  by applying the algorithm from  \cite{bian2016}, we obtain the following VI Algorithm \ref{vialg1} to approximate the solution to \eqref{Ricca}.

     \begin{algorithm}
     	\caption{VI  Algorithm for solving \eqref{Ricca} \cite{bian2016}}\label{vialg1}
     	\begin{algorithmic}
     		\State Choose $P_0=P_0^T>0$. $k,q\leftarrow 0$.
     		\Loop
     		\State $\tilde{P}_{k+1}\leftarrow P_k+\epsilon_k (Y^TP_k+P_kY-P_kJJ^TP_k+I_{n+n_z})$
     		\If{$\tilde{P}_{k+1}\notin B_q$}
     		\State $P_{k+1}\leftarrow P_0$. $q\leftarrow q+1$.
     		\ElsIf{$|\tilde{P}_{k+1}-P_k|/\epsilon_k<\varepsilon$}
     		\State \Return $P_k$ as an approximation to $P^*$
     		\Else
     		\State $P_{k+1}\leftarrow \tilde{P}_{k+1}$
     		\EndIf
     		\State $k \leftarrow k+1$
     		\EndLoop
     	\end{algorithmic}
     \end{algorithm}

     By  Theorem 3.3 of \cite{bian2016}, 
     	Algorithm \ref{vialg1} is such that $\lim\limits_{k\to \infty}P_k=P^*$.

     For simplicity, define $H_k=Y^TP_k+P_kY, K_k=-J^TP_k$. From \eqref{dyhatxirho}, we have
     \begin{align}\label{intvihat}
     	|\hat{\xi}_i&(t+\delta t)|_{P_k}-|\hat{\xi}_i(t)|_{P_k}\notag \\=&\int_{t}^{t+\delta t}\lbrack |\hat{\xi}_i|_{H_k}-2u_i^TK_k\hat{\xi}_i+ 2\eta_i^T\begin{bmatrix}
     		E_i\\0
     	\end{bmatrix}^TP_k\hat{\xi}_i  \notag \\
     	&+2\hat{\gamma}_i^T P_k\hat{\xi}_i \rbrack d\tau
     \end{align}

     \begin{rem}
     	The definition of $K_k$ in VI-based approach is slightly different from that of PI-based approach. In Section \ref{sec3-1}, $K_{k+1}\triangleq-J^TP_k$. In this section, we have $K_k\triangleq -J^TP_k$.
     \end{rem}

     Thus, \eqref{intvihat} together with \eqref{newdefi} gives
     \begin{align}\label{vilinearhat}
     	\hat{\Psi}_{VI,i}\begin{bmatrix}
     		\text{vecs}(H_k)\\
     		\text{vec}(K_k) \\
     		\text{vec}(\begin{bmatrix}
     			E_i\\0
     		\end{bmatrix}^TP_k)
     	\end{bmatrix}=\hat{\Phi}_{VI, ik}
     \end{align}
     where
     \begin{align*}
     	\hat{\Psi}_{VI,i}=&[\hat{\Psi}_{VI,i}^1, \hat{\Psi}_{VI,i}^2, \hat{\Psi}_{VI,i}^3]\\
     	\hat{\Psi}_{VI,i}^1=&\Gamma_{\hat{\xi}_i\hat{\xi}_i}M\\
     	\hat{\Psi}_{VI,i}^2=&-2\Gamma_{\hat{\xi}_iu_i}\\
     	\hat{\Psi}_{VI,i}^3=&2\Gamma_{\hat{\xi}_i\eta_i}\\
     	\hat{\Phi}_{VI, ik}=&[\delta_{\hat{\xi}_i}-2\Gamma_{\hat{\xi}_i\hat{\gamma}_i}M]\text{vecs}(P_k)
     \end{align*}

     The solvability of \eqref{vilinearhat} is guaranteed by the following lemma.
     \begin{lem} \label{lem4}
     	The matrix $\hat{\Psi}_{VI,i}$ in \eqref{vilinearhat} has full column rank if
     	\begin{align}\label{rankconvihat}
     		\textup{rank}([\Gamma_{\hat{\xi}_i\hat{\xi}_i}, \Gamma_{\hat{\xi}_iu_i}, \Gamma_{\hat{\xi}_i\eta_i}])=& \frac{(n+n_z)(n+n_z+1)}{2}\notag \\
     		&+(n+n_z)(m+q)
     	\end{align}
     \end{lem}
     \begin{proof}
     	Given any column vector $a\in \mathbb{R}^{n+n_z}$ and any symmetric matrix $Q\in \mathbb{R}^{(n+n_z)\times (n+n_z)}$, it can be verified that $a^TQa=(a\otimes a)^T\text{vec}(Q) = \text{vecv}(a)^T\text{vecs}(Q)$. Further, recalling that $\text{vec}(Q)=M\text{vecs}(Q)$ from the definition of matrix $M$ gives
     	$$(a\otimes a)^TM\text{vecs}(Q)=\text{vecv}(a)^T\text{vecs}(Q)$$
     	Since $Q$ can be any symmetric matrix, we obtain $(a\otimes a)^TM=\text{vecv}(a)^T$ for any column vector $a\in \mathbb{R}^{n+n_z}$.
     	Thus, $\Gamma_{\hat{\xi}_i\hat{\xi}_i}M=\hat{\Gamma}_{\hat{\xi}_i}$, where $\hat{\Gamma}_{\hat{\xi}_i}=[\int_{t_0}^{t_1}\text{vecv}(\hat{\xi}_i) d\tau , \int_{t_1}^{t_2}\text{vecv}(\hat{\xi}_i) d\tau, \cdots ,\\ \int_{t_{s-1}}^{t_s}\text{vecv}(\hat{\xi}_i) d\tau]^T$.
     	
     	Since
     	\begin{align*}
     		\hat{\Psi}_{VI,i}&=[\Gamma_{\hat{\xi}_i\hat{\xi}_i}M, -2\Gamma_{\hat{\xi}_iu_i}, 2\Gamma_{\hat{\xi}_i\eta_i}]\\
     		&=[\hat{\Gamma}_{\hat{\xi}_i}, \Gamma_{\hat{\xi}_iu_i}, \Gamma_{\hat{\xi}_i\eta_i}]\Theta
     	\end{align*}
     	where
     	\begin{align*}
     		\Theta=\begin{bmatrix}
     			I_{\frac{(n+n_z)(n+n_z+1)}{2}}&~&~\\~&-2I_{m(n+n_z)}&~\\~&~&2I_{q(n+n_z)}
     		\end{bmatrix},
     	\end{align*}
     	 $\text{rank}(\hat{\Psi}_{VI,i})=\text{rank}([\hat{\Gamma}_{\hat{\xi}_i}, \Gamma_{\hat{\xi}_iu_i}, \Gamma_{\hat{\xi}_i\eta_i}])$. Since, by inspection,  $\text{rank}([\hat{\Gamma}_{\hat{\xi}_i}, \Gamma_{\hat{\xi}_iu_i}, \Gamma_{\hat{\xi}_i\eta_i}])=\text{rank}([\Gamma_{\hat{\xi}_i\hat{\xi}_i}, \Gamma_{\hat{\xi}_iu_i}, \Gamma_{\hat{\xi}_i\eta_i}])$. Thus, $\text{rank}(\hat{\Psi}_{VI,i})=\text{rank}([\Gamma_{\hat{\xi}_i\hat{\xi}_i}, \Gamma_{\hat{\xi}_iu_i}, \Gamma_{\hat{\xi}_i\eta_i}])$. In conclusion, \eqref{rankconvihat} means $\hat{\Psi}_{VI,i}$ has full column rank.
     \end{proof}

      In Algorithm \ref{vialg1}, $\tilde{P}_{k+1}$ can be expressed as $\tilde{P}_{k+1}\triangleq P_k+\epsilon_k (H_k-K_k^TK_k+I_{n+n_z})$.  Thus, by iteratively applying the solution of \eqref{vilinearhat} to Algorithm \ref{vialg1}, one can obtain the approximate solution to \eqref{Ricca} as well as the feedback control gain  \eqref{gainrevi} without using the information of system matrices.

       \begin{rem}
     	We further compare our approach with the one in
     		\cite{gao2021reinforcement}, which  derived the following equations from \eqref{gaoaugmented} for their VI algorithm:
     		\begin{align}\label{gaovilinear}
     		{\Psi}_{VI,i}\begin{bmatrix}
     			\text{vecs}(H_{k})\\
     			\text{vec}(K_{k})\\
     			\text{vec}(\begin{bmatrix}
     				E_i\\\alpha_{i0}G_2F
     			\end{bmatrix}^TP_k)\\
     			\text{vec}(\Xi_{ik})
     		\end{bmatrix}={\Phi}_{VI,ik}
     		\end{align}
     		where
     		\begin{align*}
     			{\Psi}_{VI,i}=&[\hat{\Gamma}_{{\xi}_i}, -2\Gamma_{{\xi}_iu_i},2\Gamma_{{\xi}_i\eta_i}, -2\Gamma_{{\xi}_i\psi_i}]\\
     			{\Phi}_{VI,ik}=&\delta_{\xi_i}\text{vecs}(P_k)
     		\end{align*}
     		 Again, since the equations in \eqref{gaoaugmented} are coupled, \eqref{gaovilinear} is much more complicated than \eqref{vilinearhat}.
     		In particular,
     		the number of unknown variables in \eqref{gaovilinear} is $\frac{(n+n_z)(n+n_z+1)}{2} + (n+n_z)(m+q+h_i)$, while our equation  \eqref{vilinearhat} only contains  $\frac{(n+n_z)(n+n_z+1)}{2} + (n+n_z)(m+q)$ unknown variables.  Thus, our method has  reduced $(n+n_z)h_i$  unknown variables.
     		Moreover,  the solvability condition for \eqref{gaovilinear}  is
     		\begin{align*}
     			&\textup{rank}([\Gamma_{{\xi}_i{\xi}_i}, \Gamma_{{\xi}_iu_i},  \Gamma_{{\xi}_i\eta_i}, \Gamma_{{\xi}_i\psi_i}])\\=&\frac{(n+n_z)(n+n_z+1)}{2}
     			+(n+n_z)(m+q+h_i)
     		\end{align*}
     		Thus, checking these conditions is much more tedious than checking the rank condition \eqref{rankconvihat}.
     		
     \end{rem}

     To obtain a concrete comparison, again consider the example with $n=10,m=8,p=4,q=20,n_z=40, N = 5$. Assume each follower has  $2$ neighbors, i.e. $|\mathcal{N}_i^-|=2$ and $h_i=2(n+n_z)=100$. TABLE \ref{tab11vi} and TABLE \ref{tab22vi} show, for each follower,
     the numbers of unknown variables and rank conditions of the two VI-based algorithms, respectively.  The contrast is stark.

     \begin{table}[!ht]
     	\caption{Comparison of Computational Complexity}\label{tab11vi}
     	\center
     	\begin{tabular}{|c|c|c|}\hline

     		Algorithm & Equation Number &Unknown variables \\ \hline
     		\cite{gao2021reinforcement}'s VI Algorithm & \eqref{gaovilinear}	& 7675 \\  \hline
     		
     		Our VI Algorithm &\eqref{vilinearhat} & 2675 \\ \hline
     		
     	\end{tabular}
     	
     \end{table}

     \begin{table}[!ht]
     	\caption{Comparison of rank condition}\label{tab22vi}
     	\center
     	\begin{tabular}{|c|c|}\hline

     		Algorithm & Rank Condition for Data Collecting \\ \hline
     		\cite{gao2021reinforcement}'s VI Algorithm & 	$\textup{rank}([\Gamma_{{\xi}_i{\xi}_i}, \Gamma_{{\xi}_iu_i},  \Gamma_{{\xi}_i\eta_i}, \Gamma_{{\xi}_i\psi_i}])=7675$
     		\\ \hline
     		
     		Our VI Algorithm& $\textup{rank}([\Gamma_{\hat{\xi}_i\hat{\xi}_i}, \Gamma_{\hat{\xi}_iu_i}, \Gamma_{\hat{\xi}_i\eta_i}])= 2675$   \\  \hline
     		
     	\end{tabular}
     	
     \end{table}

   \section{Further Improvement over PI and VI Algorithms}\label{sec4}

  It can be seen that both PI and VI algorithms boil down to solving various sequences of linear equations, whose computational cost may be formidable for a large MAS.
  Thus, it is interesting to  further improve our data-driven algorithms in Section \ref{sec3}.

   \subsection{An Improved PI-based Algorithm}\label{sec4-1}

   Since,  from \eqref{iteraeq2}, $K_{k+1}=-J^TP_k$  where $J=\begin{bmatrix}
   	B\\0
   \end{bmatrix}$, it is possible to further improve our PI-based algorithm  over  \eqref{pilinearhat} as follows.

   \begin{enumerate}
   	\item From \eqref{intequahat}, by letting $P_k=I_{n+n_z}$, we obtain
   	\begin{align} \label{intequahatnew}
   		|\hat{\xi}_i&(t+\delta t)|_{I_{n+n_z}}-|\hat{\xi}_i(t)|_{I_{n+n_z}}\notag \\=&\int_{t}^{t+\delta t}\lbrack |\hat{\xi}_i|_{Y^T+Y}+ 2u_i^TJ^T\hat{\xi}_i \notag+2\eta_i^T\begin{bmatrix}
   			E_i\\0
   		\end{bmatrix}^T\hat{\xi}_i  \\&+2\hat{\gamma}_i^T\hat{\xi}_i  \rbrack d\tau\notag \\
   		=& \int_{t}^{t+\delta t}\lbrack |\hat{\xi}_i|_{Y^T+Y}+ 2u_i^TB^Tx_i \notag +2\eta_i^TE_i^Tx_i  \\&+2\hat{\gamma}_i^T\hat{\xi}_i  \rbrack d\tau
   	\end{align}
   	
   	Combining \eqref{newdefi} and \eqref{intequahatnew} gives
   	\begin{align}\label{pilinearnew}
   		\hat{\Psi}_{0}\begin{bmatrix}
   			\text{vecs}(Y^T+Y)\\
   			\text{vec}(B^T)\\
   			\text{vec}(E_i^T)
   		\end{bmatrix}=\hat{\Phi}_{0}
   	\end{align}
   	where
   	\begin{align*}
   		\hat{\Psi}_{0}=&[\Gamma_{\hat{\xi}_i\hat{\xi}_i}M, 2\Gamma_{x_iu_i},2\Gamma_{x_i\eta_i}]\\
   		\hat{\Phi}_{0}=&(\delta_{\hat{\xi}_i}-2\Gamma_{\hat{\xi}_i\hat{\gamma}_i}M)\text{vecs}(I_{n+n_z})
   	\end{align*}
   Solving \eqref{pilinearnew} gives $B$ and $E_i$, and hence $J$ and $\begin{bmatrix}
   		E_i\\0
   	\end{bmatrix}$.
   	\item Since $K_{k+1}=-J^TP_k$,   with  $J$ and $\begin{bmatrix}
   		E_i\\0
   	\end{bmatrix}$ known,  \eqref{pilinearhat} can be reduced to the following
   	\begin{align}\label{pinew2}
   		\hat{\Psi}_{PI,ik}'\begin{bmatrix}
   			\text{vecs}(P_{k})
   		\end{bmatrix}=\hat{\Phi}_{PI,ik}'
   	\end{align}
   	where
   	\begin{align*}
   		\hat{\Psi}_{PI,ik}'=&\delta_{\hat{\xi}_i}-2\Gamma_{\hat{\xi}_{i}\hat{\gamma}_i}M\\
   		&+2\Gamma_{\hat{\xi}_i\hat{\xi}_i}(I_{n+n_z}\otimes K_k^TJ^T)M\\
   		&-2\Gamma_{\hat{\xi}_iu_i}(I_{n+n_z}\otimes J^T)M\\
   		&-2\Gamma_{\hat{\xi}_i\eta_i}(I_{n+n_z}\otimes \begin{bmatrix}
   			E_i\\0
   		\end{bmatrix}^T)M\\
   		\hat{\Phi}_{PI,ik}'=& -\Gamma_{\hat{\xi}_i\hat{\xi}_i}\text{vec}(I_{n+n_z}+K_k^TK_k)
   	\end{align*}
   	Solving \eqref{pinew2} gives  $P_k$ and $K_{k+1}=-J^TP_k$.
   \end{enumerate}

   We further establish the solvability of  \eqref{pilinearnew} and \eqref{pinew2} by the following two lemmas.
   \begin{lem} \label{lem1}
   	The matrix $\hat{\Psi}_{0}$ in \eqref{pilinearnew} has full column rank  if
   	\begin{align}\label{rankconpinew}
   		\textup{rank}([\Gamma_{\hat{\xi}_i\hat{\xi}_i}, \Gamma_{x_iu_i}, \Gamma_{x_i\eta_i}])=&\frac{(n+n_z)(n+n_z+1)}{2}\notag \\
   		&+n(m+q)
   	\end{align}
   \end{lem}

   \begin{proof}
   	The proof of this lemma is similar to that of Lemma \ref{lem4} and is thus omitted.
   \end{proof}

   \begin{lem}
   	The matrix $\hat{\Psi}_{PI,ik}'$ in \eqref{pinew2} has full column rank for any $k\in \mathbb{N}$ if
   	\begin{align}\label{rankconpi2}
   		\textup{rank}([\Gamma_{\hat{\xi}_i\hat{\xi}_i}])=&\frac{(n+n_z)(n+n_z+1)}{2}
   	\end{align}
   \end{lem}
   \begin{proof}
   	The proof of this lemma is similar to that of Lemma \ref{lem3-1} and is thus omitted.
   \end{proof}

   \begin{rem}\label{rem11}
   	 Compared with \eqref{pilinearhat}, the number of the unknown variables of \eqref{pinew2} is reduced by $(n+n_z)(m+q)$. Thus, our improved PI-based Algorithm has significantly reduced the computational cost of the PI-based algorithm in Section \ref{sec3-1}. Moreover,
    The rank condition \eqref{rankconpi2} is satisfied automatically once \eqref{rankconpinew} is. Thus, \eqref{rankconpinew} is the only rank requirement for the improved PI-based Algorithm in this subsection.  Since the column dimension of the matrix to be tested is reduced by $n_z(m+q)$, the rank condition \eqref{rankconpinew} is much milder than that of \eqref{rankconpihat},
   \end{rem}

   Our improved PI-based Algorithm is summarized as Algorithm \ref{algpi}.

   \begin{algorithm}
   	\caption{The Improved PI-based Algorithm  for CORP}\label{algpi}
   	\begin{algorithmic}[1]
   		\State $i\leftarrow 1$
   		\Repeat
   		\State Find a feedback gain $K_{0}$ such that $Y+JK_0$ is a Hurwitz matrix. Apply an initial input $u_i^0=K_{0}\hat{\xi}_i+\delta_i$ with exploration noise $\delta_i$. Choose a proper $t_0$ such that $|\rho_i|$ is small enough.
   		\State Collecting data from $t_0$ until the rank condition \eqref{rankconpinew} is satisfied.
   		\State Solve $B$ and $E_i$ from \eqref{pilinearnew}. Obtain $J$ and $\begin{bmatrix}
   			E_i\\0
   		\end{bmatrix}$.
   		\Repeat
   		\State Solve $P_{k}$ from (\ref{pinew2}). $K_{k+1}=-J^TP_k$. $k\leftarrow k+1$.
   		\Until{$|P_{k}-P_{k-1}|<\varepsilon$ with a sufficiently small constant $\varepsilon>0$}
   		\State $K^*=[K_x^*,K_z^*]=\omega^{-1}K_k$ with $0<\omega \leq 2\text{Re}(\lambda_i), \forall \lambda_i\in \sigma(DH), i=1,\cdots,N$.
   		\State Obtain the following controller
   		\begin{align}\label{controllerpi}
   			u_i^*=\frac{K_{x}^*}{\sum_{j\in \mathcal{N}_i} a_{ij}}(\sum_{j\in \mathcal{N}_i^-}a_{ij}(x_i-x_j)+a_{i0}x_i)+K_{z}^*\hat{z}_i
   		\end{align}
   		\State $i=i+1$
   		\Until{$i=N+1$}
   	\end{algorithmic}
   \end{algorithm}

   For comparison of the PI algorithm in Section \ref{sec3-1} and the improved PI algorithm in this subsection, let us still consider the example with $n=10,m=8,p=4,q=20,n_z=40, N= 5$.  TABLE \ref{tab1} and TABLE \ref{tab2} show
   the numbers of unknown variables and rank conditions of these two algorithms, respectively.  The contrast is stark.

   \begin{table}[!ht]
   	\caption{Comparison of Computational Complexity}\label{tab1}
   	\center
   	\begin{tabular}{|c|c|c|}\hline

   		Algorithm & Equation Number &Unknown variables \\ \hline
   		PI Algorithm & \eqref{pilinearhat}	& 2675 \\  \hline
   		
   		Improved PI Algorithm &\eqref{pinew2} & 1275 \\ \hline
   		
   	\end{tabular}
   	
   \end{table}

   \begin{table}[!ht]
   	\caption{Comparison of rank condition}\label{tab2}
   	\center
   	\begin{tabular}{|c|c|}\hline

   		Algorithm & Rank Condition for Data Collecting \\ \hline
   		PI Algorithm & 	
   		$\textup{rank}([\Gamma_{\hat{\xi}_i\hat{\xi}_i}, \Gamma_{\hat{\xi}_iu_i}, \Gamma_{\hat{\xi}_i\eta_i}])= 2675$ \\ \hline
   		
   		Improved PI Algorithm& $ \textup{rank}([\Gamma_{\hat{\xi}_i\hat{\xi}_i}, \Gamma_{x_iu_i}, \Gamma_{x_i\eta_i}])=1555$   \\  \hline
   		
   	\end{tabular}
   	
   \end{table}

   \subsection{An Improved VI-based Algorithm}\label{sec4-2}

    Similar to PI algorithm,  we can also further improve our VI-based Algorithm in Section \ref{sec3-2} by the following two steps.
   \begin{enumerate}
   	\item Solve \eqref{pilinearnew} for $B$ and $E_i$ and hence obtain $J$ and $\begin{bmatrix}
   		E_i\\0
   	\end{bmatrix}$.
   	
   	\item With $J$ and $\begin{bmatrix}
   		E_i\\0
   	\end{bmatrix}$ identified, \eqref{vilinearhat} can be reduced to
   	\begin{align}\label{vinew2}
   		\hat{\Psi}_{VI,i}'\begin{bmatrix}
   			\text{vecs}(H_k)
   		\end{bmatrix}=&\hat{\Phi}_{VI, ik}'
   	\end{align}
   	where
   	\begin{align*}
   		\hat{\Psi}_{VI,i}'=&\Gamma_{\hat{\xi}_i\hat{\xi}_i}M\\
   		\hat{\Phi}_{VI, ik}'=&[\delta_{\hat{\xi}_i}-2\Gamma_{\hat{\xi}_i\hat{\gamma}_i}M\\&-2\Gamma_{\hat{\xi}_iu_i}(I_{n+n_z}\otimes J^T)M\\&-2\Gamma_{\hat{\xi}_i\eta_i}(I_{n+n_z}\otimes \begin{bmatrix}
   			E_i\\0
   		\end{bmatrix}^T)M]\text{vecs}(P_k)
   	\end{align*}
   	In this way, $H_k$ can be iteratively solved from  \eqref{vinew2}.
   \end{enumerate}

 As in Algorithm \ref{vialg1}, $\tilde{P}_{k+1}$ can be expressed as $\tilde{P}_{k+1}\triangleq P_k+\epsilon_k (H_k-P_kJJ^TP_k+I_{n+n_z})$.  Thus, by iteratively applying the solution of \eqref{vinew2} to Algorithm \ref{vialg1}, one can obtain the approximate solution to \eqref{Ricca} and the feedback control gain  \eqref{gainrevi} without using the information of system matrices.

   The solvability of \eqref{vinew2} is guaranteed by the following lemma.
   \begin{lem}
   	The matrix $\hat{\Psi}_{VI,i}'$ in \eqref{vinew2} has full column rank if
   	\begin{align}\label{rankconvi2}
   		\textup{rank}([\Gamma_{\hat{\xi}_i\hat{\xi}_i}])=&\frac{(n+n_z)(n+n_z+1)}{2}
   	\end{align}
   \end{lem}
   \begin{proof}
   	The proof of this lemma is similar to that of Lemma \ref{lem4} and is thus omitted.
   \end{proof}

   \begin{rem}
   	Since \eqref{vinew2} is solvable if the rank condition \eqref{rankconpinew} is ensured,  \eqref{rankconpinew} is the only rank requirement for our VI-based Algorithm in this subsection.  The rank condition \eqref{rankconpinew} is milder than \eqref{rankconvihat} as the column dimension of the matrix to be tested is reduced by $n_z(m+q)$.
   \end{rem}

   \begin{rem}
   	Since, in our improved VI-based Algorithm, $H_k$ is iteratively solved from \eqref{vinew2}, the number of the unknown variables of \eqref{vinew2} is reduced by $(n+n_z)(m+q)$ compared with \eqref{vilinearhat}. Thus, the improved VI-based Algorithm  significantly reduces the computational cost of the original algorithm.
   \end{rem}

\begin{rem}\label{new11}
Another merit of our improved PI and VI algorithms is that  the unknown variables of either \eqref{pinew2} or \eqref{vinew2} are independent of the index $i$.
By taking advantage of the connectivity of the graph,  there is no need to solve  \eqref{pinew2} or \eqref{vinew2} for all $i = 1, \cdots, N$. For example, 
let $\mathcal{G}^-=\{\mathcal{V}^-, \mathcal{E}^- \}$ be a subgraph of $\mathcal{G}$ with
$\mathcal{V}^-=\{1,\cdots, N \}$  and $\mathcal{E}^-$ being obtained from $\mathcal{E}$ by removing all nodes incident to node $0$. Suppose $\mathcal{G}^-$ contains a spanning tree with one of the children of $\mathcal{V}^-$ as its root. Then it suffices to solve \eqref{pinew2} or \eqref{vinew2} with respect to some index $1 \leq i \leq N$ corresponding to the root of $\mathcal{G}^-$. Then the solution of \eqref{pinew2} or \eqref{vinew2} can be passed to other followers by the graph $\mathcal{G}^-$.
 \end{rem}

   Our improved VI-based algorithm is summarized as Algorithm \ref{vialg2}.

   \begin{algorithm}
   	\caption{The Improved VI-based  Algorithm for CORP}\label{vialg2}
   	\begin{algorithmic}[1]
   		\State $i=1$
   		\Repeat
   		\State Choose $P_0=P_0^T> 0$. $k,q\leftarrow 0$. Apply any locally essentially bounded initial input $u_i^0$. Choose a proper $t_0$ such that $|\rho_i|$ is small enough.
   		\State Collecting data from $t_0$ until the rank condition \eqref{rankconpinew} is satisfied.
   		\State Solve $B$ and $E_i$ from \eqref{pilinearnew}. Obtain $J$ and $\begin{bmatrix}
   			E_i\\0
   		\end{bmatrix}$.
   		\Loop
   		\State Solve $H_k$ from \eqref{vinew2}
   		\State $\tilde{P}_{k+1}\leftarrow P_k+\epsilon_k (H_k-P_kJJ^TP_k+I_{n+n_z})$
   		\If{$\tilde{P}_{k+1}\notin B_q$}
   		\State $P_{k+1}\leftarrow P_0$. $q\leftarrow q+1$.
   		\ElsIf{$|\tilde{P}_{k+1}-P_k|/\epsilon_k<\varepsilon$}
   		\State \Return $K_k=-J^TP_k$
   		\Else
   		\State $P_{k+1}\leftarrow \tilde{P}_{k+1}$
   		\EndIf
   		\State $k \leftarrow k+1$
   		\EndLoop
   		\State $K^*=[K_x^*,K_z^*]=\omega ^{-1}K_k$ with $0<\omega \leq 2\text{Re}(\lambda_i), \forall \lambda_i\in \sigma(DH), i=1,\cdots,N$.
   		\State Obtain the following controller
   		\begin{align}\label{controllervi}
   			u_i^*=\frac{K_{x}^*}{\sum_{j\in \mathcal{N}_i} a_{ij}}(\sum_{j\in \mathcal{N}_i^-}a_{ij}(x_i-x_j)+a_{i0}x_i)+K_{z}^*\hat{z}_i
   		\end{align}
   		\State $i=i+1$
   		\Until{$i=N+1$}
   	\end{algorithmic}
   \end{algorithm}

   To compare the VI algorithm in Section \ref{sec3-2} and the improved VI algorithm in this subsection, let us still consider the example with $n=10,m=8,p=4,q=20,n_z=40, N= 5$.  TABLE \ref{tab1vi} and TABLE \ref{tab2vi} show
   the numbers of unknown variables and rank conditions of these two VI-based algorithms, respectively.

   \begin{table}[!ht]
   	\caption{Comparison of Computational Complexity}\label{tab1vi}
   	\center
   	\begin{tabular}{|c|c|c|}\hline

   		Algorithm & Equation Number &Unknown variables \\ \hline
   		VI Algorithm & \eqref{vilinearhat}	& 2675 \\  \hline
   		
   		Improved VI Algorithm &\eqref{vinew2} & 1275 \\ \hline
   		
   	\end{tabular}
   	
   \end{table}

   \begin{table}[!ht]
   	\caption{Comparison of rank condition}\label{tab2vi}
   	\center
   	\begin{tabular}{|c|c|}\hline

   		Algorithm & Rank Condition for Data Collecting \\ \hline
   		VI Algorithm & 	
   		$\textup{rank}([\Gamma_{\hat{\xi}_i\hat{\xi}_i}, \Gamma_{\hat{\xi}_iu_i}, \Gamma_{\hat{\xi}_i\eta_i}])= 2675$ \\ \hline
   		
   		Improved VI Algorithm& $ \textup{rank}([\Gamma_{\hat{\xi}_i\hat{\xi}_i}, \Gamma_{x_iu_i}, \Gamma_{x_i\eta_i}])=1555$   \\  \hline
   		
   	\end{tabular}
   	
   \end{table}

	\section{Conclusions}\label{sec6}
	
	In this paper, we have investigated the cooperative output regulation problem for unknown linear homogeneous MASs using data-driven distributed internal model approach.
	Compared with existing results in the literature,
	this paper has offered four  features. First, by leveraging the technique from \cite{su2012general}, {our approach has successfully solved the problem
for MASs whose followers are governed by MIMO unknown linear systems. Second, we have removed the restrictive assumption that  the communication network is acyclic.
Third, by using the normalized virtual error and lumping all the neighbors’ output into one single term, we have derived a set of much simpler equations
and obtained milder rank conditions compared to the algorithms of \cite{gao2021reinforcement}. Fourth, we have further improved our algorithms by
reducing high dimensional equations to lower dimensional equations.

%
%
%
%
%

\end{document}